\documentclass[12pt]{article}
 \usepackage{amsmath,amssymb,amsthm}
 \RequirePackage[dvips]{graphicx}
 \usepackage{graphics,color}
 \def\draw #1 by #2 (#3){
  \vbox to #2{
    \hrule width #1 height 0pt depth 0pt
    \vfill
    \special{picture #3} 
    }
  }

 \def\scaleddraw #1 by #2 (#3 scaled #4){{
  \dimen0=#1 \dimen1=#2
  \divide\dimen0 by 1000 \multiply\dimen0 by #4
  \divide\dimen1 by 1000 \multiply\dimen1 by #4
  \draw \dimen0 by \dimen1 (#3 scaled #4)}
  }

\usepackage{hyperref}

\newtheorem{theorem}{Theorem}[section]
\newtheorem{example}{Example}
\newtheorem{problem}[example]{Problem}
\newtheorem{defin}[theorem]{Definition}
\newtheorem{lemma}[theorem]{Lemma}
\newtheorem{conjecture}[theorem]{Conjecture}
\newtheorem{corollary}[theorem]{Corollary}
\newtheorem{remark}[theorem]{Remark}
\newtheorem{nt}{Note}
\newtheorem{claim}{Claim}

\newenvironment{pf}{\medskip\noindent{Proof:  \hspace*{-.4cm}}
       \enspace}{\hfill \qed \newline \medskip}

 \newcommand{\singlespacing}{\let\CS=\@currsize\renewcommand{\baselinestretch}{1}\tiny\CS}
 \newcommand{\oneandahalfspacing}{\let\CS=\@currsize\renewcommand{\baselinestretch}{1.25}\tiny\CS}
 \newcommand{\doublespacing}{\let\CS=\@currsize\renewcommand{\baselinestretch}{1.35}\tiny\CS}

 \newtheorem{rule-def}[theorem]{Rule}

\begin{document}
\baselineskip 16pt
 \newcommand{\la}{\lambda}
 \newcommand{\si}{\sigma}
 \newcommand{\ol}{1-\lambda}
 \newcommand{\be}{\begin{equation}}
 \newcommand{\ee}{\end{equation}}
 \newcommand{\bea}{\begin{eqnarray}}
 \newcommand{\eea}{\end{eqnarray}}

\begin{center}
 {\Large \bf Sum of the $k$ Largest Eigenvalues of Symmetric \\[1mm] Matrices: Theory and Applications}\\

  \vspace{10mm}

 {\large \bf Shaowei Sun$^{a,\,b}$, Yaping Min$^{a}$, Kinkar Chandra Das$^{c,}\footnote{Corresponding author}$}

 \vspace{9mm}

 \baselineskip=0.20in
$^a${\it School of Science, Zhejiang University of Science and Technology,\\
 Hangzhou, Zhejiang, 310023, PR China \/}\\
{\rm E-mail:} {\tt sunshaowei2009@126.com,\,1915068342@qq.com}\\[2mm]
$^b${\it School of Mathematics, East China University of Science and Technology,\\
 Shanghai, 200237, PR China \/}\\[2mm]
$^c${\it Department of Mathematics, Sungkyunkwan University,\\
Suwon 16419, Republic of Korea\/}\\[2mm]
{\rm E-mail:} {\tt kinkardas2003@gmail.com}

 \vspace{4mm}

 \end{center}

 \vspace{5mm}

 \baselineskip=0.23in

 \begin{abstract}
This paper establishes new upper bounds for the sum of the $k$ largest eigenvalues of symmetric matrices. When applied to the adjacency matrix of a graph, our results improve upon a related bound due to Mohar {\bf [On the sum of k largest eigenvalues of graphs and symmetric matrices, J. Combin. Theory Ser. B 99 (2009) 306--313]}. Furthermore, in the case of the Laplacian matrix, we prove that the well-known Brouwer's conjecture {\bf [Spectra of Graphs, Springer, New York, 2012]} holds for small values of $k$ for almost all graphs, thereby taking a significant step toward its complete resolution.

 \bigskip

 \end{abstract}

 \baselineskip=0.30in

 \section{Introduction}
 The spectral theory of symmetric matrices underpins vast areas of
 mathematics and its applications in the physical sciences, engineering and data
 science. The eigenvalues of a symmetric matrix  encode critical information about the fundamental properties of the system it represents.
 While computing the exact eigenvalues is a central numerical task, deriving rigorous a priori bounds for these eigenvalues constitutes
 a parallel and equally profound strand of research in matrix
 theory. In this paper, we will study the sum of the $k$ largest
 eigenvalues of symmetric matrices.

 Let $M$ be an $n\times n$ symmetric matrix, whose eigenvalues are all real and ordered as $\lambda_1(M) \geq \lambda_2(M) \geq \cdots \geq \lambda_n(M)$.
  Denote by $\nu(M)$ (respectively, $\theta(M)$) the number of eigenvalues of $M$ greater than (respectively, less than)
  $\frac{\operatorname{tr}(M)}{n}$, where $\operatorname{tr}(M)$ is the trace of $M$. When $\operatorname{tr}(M) = 0$, the quantities $\nu(M)$ and $\theta(M)$ are referred to as the positive and negative inertia of $M$, respectively.
  For a graph $G$, we write $\lambda_i(G) = \lambda_i(A(G))$ and $\mu_i(G) = \lambda_i(L(G))$, where $A(G)$ and $L(G)$ denote the adjacency matrix and Laplacian matrix of $G$, respectively. We use $S_M(G)$ to denote the spectrum of the matrix $M$ of $G$.

 In \cite{MOHAR}, Mohar obtained an upper bound on the sum of $k$ largest
 eigenvalues of a symmetric matrix via a simple approach. Nikiforov
 \cite{NIK2} strengthened Mohar's result by incorporating the singular values of
 the matrix. Inspired by these works, we further investigate upper bounds on the sum of the $k$ largest eigenvalues of symmetric matrices.

 When we consider a symmetric matrix as a graph matrix, numerous related results have been established regarding the sum of
 its $k$ largest eigenvalues--for instance, for the adjacency matrix \cite{DMS,MOHAR,NIK2}, (signless) Laplacian matrix \cite{AOT,LG}, distance matrix \cite{ZL}, and distance signless Laplacian matrix \cite{DP1}.
 In the case of the Laplacian matrix, a well-known open problem is Brouwer's conjecture, stated as follows:
  \begin{conjecture}  \label{conj1} {\rm (Brouwer's conjecture) \cite{E}}
 Let $G$ be a graph with order $n$ and size $m$. For each $k\in\{1,2,\dots,n-1\}$,
  $$\sum_{i=1}^{k}\mu_i\leq m+\binom{k+1}{2}.$$
  \end{conjecture}
 This conjecture has been verified in several special cases but remains open in general. For recent progress, we refer to \cite{CZ}. A further motivation of this paper is to apply our main
 results to the Laplacian matrix, with the aim of contributing to the study of Brouwer's conjecture.

 The rest of this paper is organized as follows. In Section 2, we
 present our main results concerning the sum of the $k$ largest
 eigenvalues of symmetric matrices. Section 3 is devoted to
 applications of these results to various graph matrices. Finally, we
 conclude the paper with a summary in Section 4.

 \section{Sum of the $k$ largest eigenvalues of a matrix}

In this section, we focus on deriving upper bounds for the sum of the $k$ largest eigenvalues of a symmetric matrix. We now present the first result.

 \begin{theorem} \label{th1} Let $M$ be an $n\times n$ symmetric matrix. Let $s$ be an integer such that $\lambda_s(M)>\frac{tr(M)}{n}$. Then for any $k\geq s$, we have
 \begin{equation}
 \sum_{i=1}^{k}\lambda_i(M)\leq\frac{k}{n}\,tr(M)+\sqrt{tr\left(\Big(M-\frac{tr(M)}{n}I\Big)^2\right)\left(k-\frac{s^2}{\theta+s}\right)}. \label{e1}
 \end{equation}
 The equality holds if and only if  $k=s=\nu$ with 
 $$\lambda_1(M)=\lambda_2(M)=\cdots=\lambda_s(M)>\frac{tr(M)}{n}\geq \lambda_{s+1}(M)~\mbox{ and }~\lambda_{n-\theta+1}(M)=\lambda_{n-\theta+2}(M)=\cdots=\lambda_n(M).$$
 \end{theorem}

 \begin{proof} Since 
 $$\sum\limits^n_{i=1}\,\lambda_i(M)=tr(M)~\mbox{ and }~\lambda_1(M)\geq \lambda_2(M)\geq\cdots\geq\lambda_n(M).$$
 It follows that $\lambda_1(M)\geq \frac{tr(M)}{n}$. If $\lambda_1(M)=\frac{tr(M)}{n}$, then all eigenvalues must be equal:
 $$\lambda_1(M)=\lambda_2(M)=\cdots=\lambda_n(M)=\frac{tr(M)}{n}.$$ 
 This implies that $s$ does not exist, which leads to a contradiction. Otherwise, $\lambda_1(M)>\frac{tr(M)}{n}$, that is, $M\neq\frac{tr(M)}{n}I$ and $s\geq 1$.
 Let $P=M-\frac{tr(M)}{n}I$.
  It is clear that
 $\lambda_i(P)=\lambda_i(M)-\frac{tr(M)}{n}$ and $tr(P)=0$. Then we set $\nu=\nu(M)=\nu(P)$ and $\theta=\theta(M)=\theta(P)$, which are defined before.
 To prove the result in (\ref{e1}), we will show the following equivalent result:
 \begin{equation}
 \sum_{i=1}^{k}\lambda_i(P)\leq \sqrt{tr(P^2)\left(k-\frac{s^2}{\theta+s}\right)}~~~\mbox{for } k\geq s, \label{e2}
 \end{equation}
 with equality holding if and only if $k=s=\nu$, $\lambda_1(P)=\lambda_2(P)=\cdots=\lambda_s(P)$ and $\lambda_{n-\theta+1}(P)=\lambda_{n-\theta+2}(P)=\cdots=\lambda_n(P)$.
 \noindent

 \vspace*{3mm}

\noindent
By Cauchy-Schwarz inequality, we obtain
\begin{align}
\sum\limits^{s}_{i=1}\lambda_i^2(P)\geq \frac{1}{s}\left(\sum\limits^{s}_{i=1}\lambda_{i}(P)\right)^2\label{1kin1}
\end{align}
with equality if and only if $\lambda_1(P)=\lambda_2(P)=\cdots=\lambda_s(P)$. Since $\nu\geq s$, using the above result, we obtain
\begin{align}
 tr(P^2)=\sum\limits^{\nu}_{i=1}\lambda_i^2(P)+\sum\limits^{\theta}_{i=1}\lambda_{n-i+1}^2(P)&\geq\sum\limits^{s}_{i=1}\lambda^2_i(P)+\sum\limits^{\theta}_{i=1}\lambda_{n-i+1}^2(P) \nonumber\\[3mm]
  &\geq \frac{1}{s}\left(\sum\limits^{s}_{i=1}\lambda_{i}(P)\right)^2+\sum\limits^{\theta}_{i=1}\lambda_{n-i+1}^2(P)\label{e3}
\end{align}
with equality if and only if $\nu=s$ and $\lambda_1(P)=\lambda_2(P)=\cdots=\lambda_s(P)$.

\vspace*{3mm}

 \noindent
 By the Cauchy-Schwarz inequality with the fact that $tr(P)=0$, we have
 \begin{equation}
 \sum\limits^{\theta}_{i=1}\lambda_{n-i+1}^2(P)\geq \frac{\left(\sum\limits^{\theta}_{i=1}\,|\lambda_{n-i+1}(P)|\right)^2}{\theta}=\frac{\left(\sum\limits^{\nu}_{i=1}\lambda_{i}(P)\right)^2}{\theta}\geq \frac{\left(\sum\limits^{s}_{i=1}\lambda_{i}(P)\right)^2}{\theta}\label{e4}
 \end{equation}
 with equality if and only if $\nu=s$ and $|\lambda_{n-\theta+1}(P)|=|\lambda_{n-\theta+2}(P)|=\cdots=|\lambda_{n}(P)|$.

 \vspace*{3mm}

 \noindent
Combining the results in (\ref{e3}) and (\ref{e4}), we get
\begin{eqnarray}
\sum\limits^{s}_{i=1}\lambda_{i}(P)\leq \sqrt{\frac{\theta\,s}{\theta+s}\,tr(P^2)}. \nonumber
\end{eqnarray}

If $k=s$, then from the above, the result in (\ref{e2}) follows. Moreover, the above equality holds if and only if $\nu=s$ with $\lambda_1(P)=\lambda_1(P)=\cdots=\lambda_s(P)$ and $\lambda_{n-\theta+1}(P)=\lambda_{n-\theta+2}(P)=\cdots=\lambda_n(P)$. 

\vspace*{3mm}

\noindent
Otherwise, $k\geq s+1$. First we can assume that $s+1\leq k\leq \nu$. Next we will prove the inequality in (\ref{e2}) holds strictly for $s+1\leq k\leq
 \nu$. Since $\nu>s$, by Cauchy-Schwarz inequality with (\ref{1kin1}) and (\ref{e4}), we obtain
 \begin{eqnarray}
 \sum_{i=1}^{k}\lambda_i(P)&\leq&\sum\limits^{s}_{i=1}\lambda_{i}(P)+\sqrt{(k-s)\sum\limits^k_{i=s+1}\lambda_i^2(P)}\nonumber\\[3mm]
 &\leq&\sum\limits^{s}_{i=1}\lambda_{i}(P)+\sqrt{\left(k-s\right)\left(tr(P^2)-\sum\limits^{\theta}_{i=1}\lambda_{n-i+1}^2(P)-\sum\limits^{s}_{i=1}\lambda^2_{i}(P)\right)}\nonumber\\[3mm]
 &<&\sum\limits^{s}_{i=1}\lambda_{i}(P)+\sqrt{\left(k-s\right)\left(tr(P^2)-\frac{1}{\theta}\left(\sum\limits^{s}_{i=1}\lambda_{i}(P)\right)^2-\frac{1}{s}\left(\sum\limits^{s}_{i=1}\lambda_{i}(P)\right)^2\right)} \nonumber\\[3mm]
  &=&\sum\limits^{s}_{i=1}\lambda_{i}(P)+\sqrt{\left(k-s\right)\left(tr(P^2)-\frac{\theta+s}{\theta\,s}\left(\sum\limits^{s}_{i=1}\lambda_{i}(P)\right)^2\right)}. \label{e6}
 \end{eqnarray}

 \noindent
 Let us consider a function
  \begin{equation}
 f(x)= x+\sqrt{(k-s)\left(tr(P^2)-\frac{\theta+s}{\theta\,s}x^2\right)},\,~~0\leq x\leq \sqrt{\frac{\theta\,s}{\theta+s}\,tr(P^2)}.\nonumber
 \end{equation}
 Then $f'(x)=0$ gives
 $$x_1= \theta\,s\sqrt{\frac{tr(P^2)}{(\theta+s)(k\,\theta+k\,s-s^2)}}<\sqrt{\frac{\theta\,s}{\theta+s}\,tr(P^2)},$$
 which implies that $f(x)$ is increasing for $x\leq x_1$, and decreasing for $x\geq x_1$.
 Thus, we have
 \begin{eqnarray}
 f(x)\leq f(x_1)=\sqrt{tr(P^2)\left(k-\frac{s^2}{\theta+s}\right)}.\nonumber
 \end{eqnarray}
Combining the above result with (\ref{e6}), we obtain
\begin{equation}
 \sum_{i=1}^{k}\lambda_i(P)<\sqrt{tr(P^2)\left(k-\frac{s^2}{\theta+s}\right)}\label{1s1}
\end{equation}
for $s+1\leq k\leq \nu$. The result in (\ref{e2}) holds strictly. 

 \vspace*{3mm}

Next we assume that $k\geq \nu+1$. Since $\lambda_i(P)\leq 0$ for $i\geq \nu+1$, we obtain
$$\sum_{i=1}^{k}\lambda_i(P)\leq \sum_{i=1}^{\nu}\lambda_i(P)<\sqrt{tr(P^2)\left(\nu-\frac{s^2}{\theta+s}\right)}<\sqrt{tr(P^2)\left(k-\frac{s^2}{\theta+s}\right)}$$
because (\ref{1s1}) holds for $k=\nu$. Consequently, the inequality in (\ref{e2}) holds strictly. This completes the proof of the theorem.
\end{proof}

 Note that the condition of $s$ in the above theorem is $s\leq \nu$. By the above theorem, we obtain a stronger result in which  the upper bound on $s$ is replaced by a larger number  as follows:
 \begin{theorem} \label{th3} Let $M$ be an $n\times n$ symmetric matrix. Let $t$ be an integer such that $t\leq n-\theta$. Then for any $k\geq t$, we have
 \begin{equation*}
 \sum_{i=1}^{k}\lambda_i(M)\leq\frac{k}{n}\,tr(M)+\sqrt{tr\left(\Big(M-\frac{tr(M)}{n}I\Big)^2\right)\left(k-\frac{t^2}{\theta+t}\right)}.
 \end{equation*}
 The equality holds if and only if $M=\frac{tr(M)}{n}I$ or $k=t$ with 
 $$\lambda_1(M)=\lambda_1(M)=\cdots=\lambda_t(M)> \frac{tr(M)}{n}\geq\lambda_{t+1}(M)~\mbox{ and }~\lambda_{n-\theta+1}(M)=\lambda_{n-\theta+2}(M)=\cdots=\lambda_n(M).$$
 \end{theorem}

 \begin{proof} It is trivial if $M=\frac{tr(M)}{n}I$. Thus, we assume that $M\neq\frac{tr(M)}{n}I$. Then we have $\lambda_1(M)>\frac{tr(M)}{n}$ as $M$ is symmetric.
 If $t\leq \nu$, Theorem \ref{th1} directly implies the required result together with the equality cases.
 Otherwise, $\nu+1\leq t\leq n-\theta$.
Setting $s=\nu$ and $k=\nu$ in (\ref{e1}), we obtain
\begin{align*}
 \sum_{i=1}^{\nu}\lambda_i(M)\leq\frac{\nu}{n}\,tr(M)+\sqrt{tr\left(\Big(M-\frac{tr(M)}{n}I\Big)^2\right)\,\frac{\theta\nu}{\theta+\nu}}\,. 
\end{align*}
Using the above result with the fact that $\lambda_{i}=\frac{tr(M)}{n}$ for $\nu+1\leq i\leq t$, we obtain
 \begin{align}
  \sum_{i=1}^{k}\lambda_i(M)&\leq\frac{k-t}{n}\,tr(M)+ \sum_{i=1}^{t}\lambda_i(M)=\frac{k-\nu}{n}\,tr(M)+\sum_{i=1}^{\nu}\lambda_i(M)\nonumber\\[3mm]
  &\leq\frac{k}{n}\,tr(M)+\sqrt{tr\left(\Big(M-\frac{tr(M)}{n}I\Big)^2\right)\frac{\theta\nu}{\theta+\nu}}\nonumber\\[3mm]
  &<\frac{k}{n}\,tr(M)+\sqrt{tr\left(\Big(M-\frac{tr(M)}{n}I\Big)^2\right)\frac{\theta\,t}{\theta+t}}\nonumber\\[3mm]
   &\leq \frac{k}{n}\,tr(M)+\sqrt{tr\left(\Big(M-\frac{tr(M)}{n}I\Big)^2\right)\left(k-\frac{t^2}{\theta+t}\right)}\nonumber
 \end{align}
as $k\geq t$. The result holds strictly. This completes the proof of the theorem.
 \end{proof}

Next we give a more simple upper bound on the sum of the $k$ largest eigenvalues of a matrix.
 \begin{theorem} \label{th4} Let $M$ be an $n\times n$ symmetric matrix. Let $s$ be an integer such that 
 $$\lambda_s(M)>\frac{tr(M)}{n}~\mbox{ and }~\sum\limits_{i=1}^s\frac{\lambda_i(M)}{s}\geq \frac{tr(M^2)+tr(M)}{n}-\frac{(tr(M))^2}{n^2}.$$ 
Then for any $k\geq s$, we have
\begin{equation}
 \sum_{i=1}^{k}\lambda_i(M)\leq\frac{k}{n}\,tr(M)+\frac{n\theta}{2(\theta+s)}+\frac{n}{2(\theta+s)}\sqrt{\theta\big(k\theta+ks-s^2\big)}.  \label{nne1}
\end{equation}
The equality holds if and only if $k=s=1$ with 
$$\lambda_1(M)=\frac{tr(M^2)+tr(M)}{n}-\frac{(tr(M))^2}{n^2}>\frac{tr(M)}{n}\geq \lambda_{2}(M)$$ 
and 
$$\lambda_{n-\theta+1}(M)=\lambda_{n-\theta+2}(M)=\cdots=\lambda_n(M).$$
\end{theorem}

\begin{proof}
If $\lambda_1(M)=\frac{tr(M)}{n}$, then $\lambda_2(M)=\lambda_n(M)=\frac{tr(M)}{n}$, which means the $s$ does not exist.
Otherwise, $\lambda_1(M)>\frac{tr(M)}{n}$, that is, $M\neq\frac{tr(M)}{n}I$ and $s\geq 1$. Let $P=M-\frac{tr(M)}{n}I$. Then we set $\nu=\nu(M)=\nu(P)$ and $\theta=\theta(M)=\theta(P)$.
 Since $tr(P^2)=tr(M^2)-\frac{(tr(M))^2}{n}$, it is equivalent to prove that for $k\geq s$,
 \begin{equation}
 \sum_{i=1}^{k}\lambda_i(P)\leq\frac{n\theta}{2(\theta+s)}+\frac{n}{2(\theta+s)}\sqrt{\theta\big(k\,\theta+k\,s-s^2\big)}~~~~\mbox{for  }\sum\limits_{i=1}^s\frac{\lambda_i(P)}{s}\geq \frac{tr(P^2)}{n},\label{ne2}
 \end{equation}
 with equality holding if and only if $k=s=1$ with $\lambda_1(P)=\frac{tr(P^2)}{n}>0\geq \lambda_2(P)$ and
 $\lambda_{n-\theta+1}(P)=\lambda_{n-\theta+2}(P)=\cdots=\lambda_n(P)$.

\vspace*{3mm}
 
 Combining Theorem \ref{th1} on $k=s$ with the condition that $\sum\limits_{i=1}^s\frac{1}{s}\lambda_i(P)\geq \frac{tr(P^2)}{n}$, we get
 $$\sum\limits_{i=1}^s\lambda_i(P)\leq\frac{n\theta}{\theta+s}\leq \frac{n\theta}{2(\theta+s)}+\frac{n}{2(\theta+s)}\sqrt{\theta\big(k\,\theta+k\,s-s^2\big)}.$$
 The first equality holds if and only if $s=\nu$, $\lambda_1(P)=\lambda_2(P)=\cdots=\lambda_s(P)=\frac{tr(P^2)}{n}$ and
 $\lambda_{n-\theta+1}(P)=\lambda_{n-\theta+2}(P)=\cdots=\lambda_n(P)$ and the second equality holds only for $k=s=1$.
 Hence we prove the case $k=s$ and together with equality cases.
 Next we will prove the inequality in (\ref{ne2}) strictly holds for $s+1\leq k\leq \nu$ as $\lambda_i(P)\leq 0$ for $i>\nu$. For this, we consider the following
two case: \vspace*{3mm}

\noindent
${\bf Case\,1.}$ $tr(P^2)\geq n\,\theta\,\sqrt{\displaystyle{\frac{tr(P^2)}{(\theta+s)(k\,\theta+k\,s-s^2)}}}=\displaystyle{\frac{n\,x_1}{s}}$ $(x_1$ is defined in the proof of Theorem \ref{th1}$)$. 

\vspace*{3mm}

In this case, $\sum\limits_{i=1}^s\lambda_i(P)\geq \frac{s}{n}tr(P^2)\geq x_1$. Recall that the function $f(x)$ defined in the proof of Theorem \ref{th1}, $f(x)$ is decreasing on $x\geq x_1$. Then from (\ref{e6}), we get
 \begin{eqnarray}
  \sum_{i=1}^{k}\lambda_i(P)&< &f\left(\sum\limits_{i=1}^s\lambda_i(P)\right)\leq f\left(\frac{s}{n}tr(P^2)\right)\nonumber\\[2mm]
  &=&\frac{s}{n}tr(P^2)+\sqrt{(k-s)\left(tr(P^2)-\frac{(tr(P^2))^2(\theta+s)\,s}{n^2\theta}\right)}\,.\label{eq1}
 \end{eqnarray}

 \vspace*{3mm}

 \noindent
 We now consider a function
 \begin{eqnarray}
 g(x)&=&\frac{s\,x}{n}+\sqrt{(k-s)\left(x-\frac{s\,(\theta+s)x^2}{n^2\theta}\right)}\,,~~~~0\leq x\leq \frac{n^2\theta}{s\,(\theta+s)}.\nonumber
 \end{eqnarray}
 Note that $g(x)$ can be rewriten as
 \begin{equation}
g(x)=\frac{s\,x}{n}+\sqrt{\frac{(s-k)(\theta+s)s}{\theta}\left(\frac{x}{n}-\frac{\theta\,n}{2s\,(\theta+s)}\right)^2+\frac{(k-s)\
     n^2\theta}{4s(\theta+s)}}. \label{sm1}
 \end{equation}
  Then we have
 \begin{equation}
 g'(x)= \frac{s}{n}+\frac{\sqrt{k-s}}{2}\times\frac{\displaystyle{1-\frac{2s(\theta+s)}{n^2\theta}x}}{\sqrt{\displaystyle{x-\frac{s(\theta+s)x^2}{n^2\theta}}}}\,,~~~0<x<\frac{n^2\theta}{s\,(\theta+s)}.\nonumber
 \end{equation}
 From $g'(x)=0$, one can easily get
 \begin{eqnarray}
  &&x_3=\frac{n^2\theta}{2s(\theta+s)}-\frac{n^2\theta}{2s(\theta+s)}\sqrt{\frac{\theta\,s}{k\theta+ks-s^2}}\nonumber\\[2mm]
  \mbox{and }&&x_4=\frac{n^2\theta}{2s(\theta+s)}+\frac{n^2\theta}{2s(\theta+s)}\sqrt{\frac{\theta\,s}{k\theta+ks-s^2}}\,.\nonumber
 \end{eqnarray}

 \noindent
 For convenience, let $x_5=\displaystyle{\frac{n^2\theta}{s(\theta+s)}}$. By the closed interval method on the continuous function $g(x)$ for $x\in \left[0,x_5\right]$, we have
 $$g(x)\leq \max\Big\{g(0),\,g(x_3),\,g(x_4),\,g_1(x_5)\Big\}.$$
 After simple calculation, we get $g(0)=0$ and $g(x_5)=\frac{n\theta}{\theta+s}$. By (\ref{sm1}), one can easily obtain that

 $$g(x_3)<g(x_4)=\frac{n\theta}{2(\theta+s)}+\frac{n}{2(\theta+s)}
 \sqrt{\frac{\theta(k\theta+ks-s^2)}{s}}.$$

 \noindent
 As $k>s$, then $\frac{k\theta+ks-s^{2}}{s}>\theta$, which follows $g(x_5)<g(x_4)$. Hence $g(x)\leq g(x_4)$.
 Applying the above results in (\ref{eq1}), we get
  \begin{eqnarray}
 &&\sum_{i=1}^{k}\lambda_i(P)< g(tr(P^2))\leq g(x_4)\leq\frac{n\theta}{2(\theta+s)}+\frac{n}{2(\theta+s)}
   \sqrt{\theta\big(k\,\theta+k\,s-s^2\big)},\nonumber
 \end{eqnarray}
 which implies that the inequality in (\ref{ne2}) strictly holds.
\vspace*{3mm}

\noindent
${\bf Case\,2.}$ $tr(P^2)<n\,\theta\,\sqrt{\displaystyle{\frac{tr(P^2)}{(\theta+s)(k\,\theta+k\,s-s^2)}}}=\displaystyle{\frac{n\,x_1}{s}}$. 

\vspace*{3mm}

As $x_1= \theta\,s\sqrt{\displaystyle{\frac{tr(P^2)}{(\theta+s)(k\,\theta+k\,s-s^2)}}}$, we have
    \begin{eqnarray}
  k\times tr(P^2)&<&\frac{n^2{\theta}^2k}{(\theta+s)(k\,\theta+k\,s-s^2)}\nonumber\\[2mm]
   &=&\frac{n^2{\theta}^2}{(\theta+s)(\theta+s-s^2/k)}\nonumber\\[2mm]
   &\leq &\frac{n^2{\theta}^2}{(\theta+s)(\theta+s/(s+1))} \label{dn1}
 \end{eqnarray}
as $k\geq s+1$. 

\vspace*{3mm}

Now let
 $$h(x)=\left(\frac{n\,x}{2(x+s)}+\frac{n}{2(x+s)}\sqrt{x\big(s\,x+x+s\big)}\right)^2-\frac{n^2x(2x-1)}{2(x+s)(x+s/(s+1))}.$$

 \vspace*{3mm}

\begin{claim}\label{c1} $h(x)>0$ for $x>0$.
\end{claim}

\vspace*{3mm}

\noindent
{\bf Proof of Claim \ref{c1}.} Now we can rewrite $h(x)$ as
   $$h(x)=\frac{n^2}{4(x+s)^2(sx+x+s)}p(x),$$ 
where
   $$ p(x)=(s^2-s-2)x^3-(2s^2-s-2)x^2+(3s^2+2s)x+2\sqrt{((s+1)x^2+sx)^3}.$$
Since $\Big((s+1)x^2+sx\Big)^3>(s+1)^2x^6$, we get 
 $$p(x)>(s^2+s)x^3-(2s^2-s-2)x^2+(3s^2+2s)x.$$
Due to the fact that 
$$(2s^2-s-2)^2-4(s^2+s)(3s^2+2s)=-8s^4 - 24s^3 - 15s^2 + 4s + 4<0$$ 
for $s\geq 1$, we conclude $(s^2+s)x^2-(2s^2-s-2)x+(3s^2+2s)$ has no real roots. By combining the above results, we obtain that $p(x)>0$ for $s\geq 1$ and $x>0$. 
Hence $h(x)>0$ for $x>0$,  which completes the proof of  ${\bf Claim\,\ref{c1}}$.
 
\vspace*{3mm}

From ${\bf Claim\,\ref{c1}}$ with (\ref{dn1}), we obtain
\begin{align}
\frac{n\,\theta}{2(\theta+s)}+\frac{n}{2(\theta+s)}\sqrt{\theta\,\Big(s\,\theta+\theta+s\Big)}&>\sqrt{\frac{n^2\theta(2\theta-1)\,(s+1)}{2(\theta+s)\Big(\theta\,(s+1)+s\Big)}}\nonumber\\[3mm]
&\geq \sqrt{k\,\left(1-\frac{1}{2\theta}\right)\,tr(P^2)}.\label{kcd1}
\end{align}

\vspace*{3mm}

Since
   $$(tr(P))^2=\left(\sum_{i=1}^{n}\lambda_i(P)\right)^2=tr(P^2)+2\sum_{i<j}\lambda_i(P)\lambda_j(P)~\mbox{ and }~tr(P)=0,$$ 
we obtain
   $tr(P^2)=-2\sum_{i<j}\lambda_i(P)\lambda_j(P)$.
Then 
    \begin{eqnarray}
 \left(\sum_{i=1}^{n}|\lambda_i(P)|\right)^2&=& tr(P^2)+2\sum_{i<j}|\lambda_i(P)\lambda_j(P)|\nonumber\\[2mm]
 &\geq &  tr(P^2)+\left|2\sum_{i<j}\lambda_i(P)\lambda_j(P)\right|\nonumber\\[2mm]
 &=&  2\,tr(P^2).\nonumber
 \end{eqnarray}
 \noindent
 Using this result, we get
  \begin{equation}
\sum_{i=1}^{\theta}\lambda_{n-i+1}^2(P)\geq
\frac{\left(\sum\limits_{i=1}^{\theta}|\lambda_{n-i+1}(P)|\right)^2}{\theta}=
\frac{\left(\sum\limits_{i=1}^{n}|\lambda_{i}(P)|\right)^2}{4\,\theta}\geq
\frac{tr(P^2)}{2\,\theta}\,.\label{e7}
 \end{equation}

 \vspace*{3mm}

 \noindent
 Using the above result with (\ref{kcd1}), we have
 \begin{eqnarray}
 \sum_{i=1}^{k}\lambda_i(P)&\leq& \sqrt{k\sum_{i=1}^{k}\lambda_i^2(P)}\nonumber\\[3mm]
     &\leq& \sqrt{k\left(tr(P^2)-\sum_{i=1}^{\theta}\lambda_{n-i+1}^2(P)\right)} \nonumber\\[3mm]
     &\leq& \sqrt{k\left(1-\frac{1}{2\,\theta}\right)\,tr(P^2)}\nonumber\\[3mm]
     &<&\frac{n\,\theta}{2(\theta+s)}+\frac{n}{2(\theta+s)}\sqrt{\theta\big(s\,\theta+\theta+s\big)}\nonumber\\[3mm]
     &\leq & \frac{n\,\theta}{2(\theta+s)}+\frac{n}{2(\theta+s)}\sqrt{\theta(k\,\theta+k\,s-s^2)}\nonumber
 \end{eqnarray}
as $k\geq s+1$. This completes the proof.
\end{proof}

\begin{theorem} \label{th2} Let $M$ be an $n\times n$ symmetric matrix. If 
$$\lambda_2(M)> \frac{tr(M)}{n}~\mbox{ and }~\lambda_1(M)+\lambda_2(M)\geq \frac{2\,tr(M^2)+2\,tr(M)}{n}-\frac{2(tr(M))^2}{n^2},$$
then for any $k\geq 2$, we have
  \begin{equation}
 \sum_{i=1}^{k}\lambda_i(M)\leq\frac{k}{n}\,tr(M)+\frac{n\theta}{2(\theta+2)}+\frac{n}{2(\theta+2)}\sqrt{\frac{\theta\big(k\theta+2k-4\big)}{2}}.  \label{nne1}
 \end{equation}
 The equality holds if and only if $k=2$ with 
 $$\lambda_1(M)=\lambda_2(M)=\frac{tr(M^2)+tr(M)}{n}-\frac{(tr(M))^2}{n^2}>\frac{tr(M)}{n}\geq \lambda_{3}(M)$$
and $\lambda_{n-\theta+1}(M)=\lambda_{n-\theta+2}(M)=\cdots=\lambda_n(M)$.
 \end{theorem}

 \begin{proof}
 Let $P=M-\frac{tr(M)}{n}I$. It is equivalent to prove that for $k\geq 2$,

 \begin{eqnarray}
 \sum_{i=1}^{k}\lambda_i(P)\leq\frac{n\theta}{2(\theta+2)}+\frac{n}{2(\theta+2)}\sqrt{\frac{\theta\big(k\,\theta+2k\,-4\big)}{2}}~~~~\mbox{for  }\frac{\lambda_1(P)+\lambda_2(P)}{2}\geq \frac{tr(P^2)}{n},\label{ne3}
 \end{eqnarray}
 with equality holding if and only if $k=2$ with $\lambda_1(P)=\lambda_2(P)=\frac{tr(P^2)}{n}>0\geq \lambda_3(P)$ and
 $\lambda_{n-\theta+1}(P)=\lambda_{n-\theta+2}(P)=\cdots=\lambda_n(P)$.
 The proof of the above result is similar as the proof of Theorem \ref{th4} by taking $s=2$. Now we explain it in detail. By a similar way as the proof of Theorem \ref{th4} with $s=2$ (before {\bf Case\,1}), we obtain that the result in (\ref{ne3}) holds for $k=2$, together with the equality cases. It remains to show the inequality in (\ref{ne3}) strictly holds
for $3\leq k\leq \nu$. If $tr(P^2)\geq \frac{nx_1}{2}$, by the result of {\bf Case\,1} in the proof of Theorem \ref{th4}, we get
 $$\sum_{i=1}^{k}\lambda_i(P)<g(x_4)=\frac{n\theta}{2(\theta+2)}+\frac{n}{2(\theta+2)} \sqrt{\frac{\theta\big(k\,\theta+2k\,-4\big)}{2}}.$$
Otherwise, $tr(P^2)<\frac{nx_1}{2}$. By the result of {\bf Case\,2} in the proof of Theorem \ref{th4}, we have
\begin{equation*}
 \sum_{i=1}^{k}\lambda_i(P)\leq \sqrt{k\,tr(P^2)\left(1-\frac{1}{2\,\theta}\right)}.
\end{equation*}
Appying the result in (\ref{dn1}) with $s=2$ to the above result, we get
\begin{equation}
 \sum_{i=1}^{k}\lambda_i(P)\leq \sqrt{\frac{n^2\theta(2\theta-1)}{2(\theta+2)(\theta+2/3)}}.\label{nsw1}
\end{equation}
 Let
 \begin{eqnarray*}
 R(x)=\left(\frac{n\,x}{2(x+2)}+\frac{n}{2(x+2)}\sqrt{\frac{x\big(3x+2\big)}{2}}\right)^2
   -\frac{n^2x(2x-1)}{2(x+2)(x+\frac{2}{3})}~~\mbox{for }x>0.\label{lnn1}
 \end{eqnarray*}

 \noindent
 We can rewrite $R(x)$ as $R(x)=\frac{n^2}{4(x+2)(3x+2)}S(x)$, where
$$ S(x)=-\frac{9}{2}x^3 - 10x^2 + 14x + 2\sqrt{\frac{(3x^2 +2x)^3}{2}}.$$
  As $(3x^2+2x)^3>9x^4(\sqrt{3}x+\frac{3}{2})^2$, we get $S(x)>(3\sqrt{6}-\frac{9}{2})x^3-(10-\frac{9}{2}\sqrt{2})x^2+14x>0$ for $x>0$. Hence $R(x)>0$ for $x>0$.
Using this fact in (\ref{nsw1}), we have

  \begin{eqnarray}
 \sum_{i=1}^{k}\lambda_i(P) &<&\frac{n\,\theta}{2(\theta+2)}+\frac{n}{2(\theta+2)}\sqrt{\frac{\theta\big(3\,\theta+2\big)}{2}}\nonumber\\[2mm]
     &\leq & \frac{n\,\theta}{2(\theta+2)}+\frac{n}{2(\theta+2)}\sqrt{\frac{\theta(k\,\theta+2k\,-4)}{2}}~~\mbox{  as } k\geq 3\,.\nonumber
 \end{eqnarray}
 This completes the proof.
 \end{proof}

\begin{remark}
 Note that Theorem \ref{th2} is better than Theorem \ref{th4} by taking
 $s=2$. However, they are incomparable in general.
\end{remark}

\section{Applications}
This section applies the main results from the previous section to two well-known graph matrices: the adjacency matrix and the Laplacian matrix. Throughout this section, all graphs are assumed to be simple, unweighted, and free of isolated vertices. As usual, $K_n$ and $K_{n_1,\,n_2,\ldots,\,n_k}$ $(\sum_{i=1}^k n_i=n)$ denote the complete graph and the complete $k$-partite graph on $n$ vertices, respectively.

\subsection{Useful results}
In this subsection, we will list some well-known results on
eigenvalues of adjacency matrix and Laplacian matrix of a graph,
which will be used in the next subsections. First there is a result
on spectral radius of a graph.

 \begin{lemma} \label{le1} {\rm \cite{CS}} Let $G$ be a graph of order $n$ and size $m$.
 Then $\lambda_1(G)\geq \frac{2\,m}{n}$ with equality holding if and
 only if $G$ is a regular graph.
 \end{lemma}

  \begin{lemma} \label{le2} {\rm \cite{P}} Let $G$ be a graph of order $n$ and positive inertia $\nu$.
 Then $\nu=1$  if and only if $G$ is a complete multipartite graph.
 \end{lemma}

  \begin{lemma} \label{cor5} {\rm \cite{B}}
    Let $G$ be a connected regular graph with exactly three distinct eigenvalues. Then $G$ is strongly regular.
  \end{lemma}

Next we present some results on the Laplacian spectrum of a graph.
  \begin{lemma} \label{cor3}  {\rm \cite{B}}
    Let $G$ be a graph with order $n$.
    Then $S_{L+aJ}(G)=\{\mu_1,\dots,\mu_{n-1},na\}$, where $J$ is a
    matrix whose entries all are ones.
  \end{lemma}

   \begin{lemma} \label{newle1} {\rm \cite{DM}} Let $G$ be a graph of order $n$ and size $m$.
 Then $\mu_2(G)\geq \frac{2m}{n}$ unless $G\cong K_{1,\,n-1}$. Moreover, the equality holds if and only
 if $G\cong K_{n/2,\,n/2}$ $(n$ is even$)$.
 \end{lemma}

   \begin{lemma} \label{newle2} {\rm \cite{MERRIS}} Let $G$ be a graph with $S_L(G)=\{\mu_1,\,\mu_2,\,\ldots,\,\mu_{n-1},\,0\}$.
 Then $S_L(\overline{G})=\{n-\mu_{n-1},\,n-\mu_{n-2},\,\ldots,\,n-\mu_{1},\,0\}$,
 where $\overline{G}$ is the complement of $G$.
 \end{lemma}
     \begin{lemma} \label{newle3} {\rm \cite{MERRIS}} Let $G$ be a graph of order $n$ and the maximum degree
   $\Delta$. If $G$ has at least one edge, then  $\mu_1\geq \Delta+1$.
 \end{lemma}
 The combination of these two results leads directly to the following.

    \begin{lemma} \label{newsundas} Let $G$ be a graph of order $n$ and the minimum degree
   $\delta$. If $G$ is not complete, then  $\mu_{n-1}\leq \delta$.
 \end{lemma}

  \begin{lemma} \label{cor4}  {\rm \cite{B}}
    Let $G$ be a graph with order $n$. Then the rank of $L(G)$ equals $n-k$, where $k$ is the number of connected components of $G$.
  \end{lemma}

  \begin{lemma} \label{2le1} {\rm \cite{CAE}} Let $G$ be a graph of order $n$ and size $m$.
 Then $\sum\limits_{i=1}^n d^2_i\leq \frac{2\,m^2}{n-1}+(n-2)\,m$.
 \end{lemma}

\subsection{Adjacency matrix}
In this subsection, we will apply the results obtained in last
section to adjacency matrix of a graph.

\begin{lemma} \label{le4}
  Let $G$ be a complete multipartite graph. Then G has three distinct eigenvalues if and only if  $G$ is complete bipartite or $G\cong K_{\underbrace{t,t,\dots,t}_p}$ with $p\geq 3$ and $t\geq 2$.
\end{lemma}

\begin{proof} We first prove the sufficiency. Let $G$ be a complete multipartite graph with three distinct eigenvalues. If the negative inertia $\theta$ of $G$ is $1$, then the rank of $G$ is $2$ as $\lambda_1(G)=\max_{i}{|\lambda_i(G)|}$ and $\sum\limits^n_{i=1}\,\lambda_i(G)=0$. It follows that $G$ is complete bipartite.
 Otherwise, $\theta(G)\geq 2$. Without loss of generality, we
 assume that $G\cong K_{n_1, n_2,\dots,n_p}$, where $n_1=\dots=n_{a_1}>n_{a_1+1}=\dots=n_{a_1+a_2}>\dots>n_{a_1+\dots+a_{s-1}+1}=
 \dots=n_{a_1+\dots+a_s}$, $\sum\limits_{i=1}^p n_i=n$ and $\sum\limits_{i=1}^s a_i=p$ with $3\leq p\leq n-1$.
 The characteristic polynomial of a complete multipartite graph is
given in \cite{A}, that is,
 \begin{eqnarray*}
 \Phi(G, x)=x^{n-p}\left(\prod_{j=1}^{p}(x+n_j)-\sum_{i=1}^{p}\,n_i\left(\prod_{j=1,j\neq
 i}^{p}(x+n_j)\right)\right).
 \end{eqnarray*}
 For convenience, let $m_i= n_{\sum_{j=1}^ia_j}$.
 From the above result, we have
 \begin{eqnarray}
 \Phi(G,x)=x^{n-p}\prod_{i=1}^{s}(x+m_i)^{a_i-1}\mathcal{F}(x)\label{sunm1},
 \end{eqnarray}
 where $\mathcal{F}(x)=\prod_{j=1}^{s}
 (x+m_j)-\sum_{i=1}^{s}a_i m_i\prod_{j=1,j\neq i}^{s}(x+m_j)$.
 For $s=1$, the required result follows immediately.
 For $s=2$, then $a_1a_2\geq 2$ as $p\geq 3$. It follows that $-m_1$ or $-m_2$ (or possible both) must be an eigenvalue of $G$. Now we have
 $\mathcal{F}(-m_1)=a_1m_1(m_1-m_2)>0$ and
 $\mathcal{F}(-m_2)=a_2m_2(m_2-m_1)<0$. This implies that $G$ has
 an eigenvalue between $-m_2$ and $-m_1$. In addition, $G$ has at
 least a positive eigenvalue and an eigenvalue zero. Combining these
 results, $G$ has at least four distinct eigenvalues, a contradiction.
 For $s\geq 3$, we have
 $$\mathcal{F}(-m_i)=-a_im_i\sum\limits_{j=1,j\neq i}^s(m_j-m_i),$$
 which implies that $\mathcal{F}(-m_i)\mathcal{F}(-m_{i+1})<0$ for $1\leq i\leq s-1$.
 Then $G$ has at least one eigenvalue between $-m_i$ and $-m_{i+1}$ for
 $1\leq i\leq 2$. In addition, $G$ has at
 least a positive eigenvalue and an eigenvalue zero. This
 contradicts to the condition that $G$ has three distinct
 eigenvalue.\\
 \noindent
 Next we prove the necessity. By (\ref{sunm1}), we have
 \begin{equation}
 S_A(K_{\underbrace{t,t,\dots,t}_p})=\Big\{t(p-1),\,\underbrace{0,\,\ldots,\,0}_{tp-p},\,\underbrace{-t,\,\ldots,
 -t}_{p-1}\Big\}.\label{sunm2}
 \end{equation}
 From this, one can see that $K_{\underbrace{t,t,\dots,t}_p}$ has
 three distinct eigenvalues if $p\geq 3$ and $t\geq 2$. Moveover, it is well known that a complete bipartite graph has three distinct eigenvalues. This completes the proof.
\end{proof}

 By taking $t=1$ in Theorem \ref{th3} with Lemma \ref{le4}, we get the following result:
 \begin{theorem} \label{coro1}
  Let $G$ be a graph of order $n\geq 2$ with negative inertia $\theta$. Then for any integer $k$,
 \begin{equation*}
 \sum_{i=1}^{k}\lambda_i(G)\leq \sqrt{2\,m\left(k-\frac{1}{\theta+1}\right)}
 \end{equation*}
 with equality holding if and only if  $G\cong K_{p,\,q}$ with $n=p+q$ and $k=1$ or $G\cong K_{\underbrace{t,\,t,\ldots,\,t}_p}$ $(n=p\,t,\,p\geq 3)$ with $k=1$.
 \end{theorem}

 \begin{proof}
 By setting $t=1$ and $M=A(G)$ in Theorem \ref{th3} with the facts that $tr(A)=0$ and $tr(A^2)=2m$, we get
 $$\sum_{i=1}^{k}\lambda_i(G)\leq\sqrt{2m\left(k-\frac{1}{\theta+1}\right)}$$
 with equality holding if and only if $k=1$ with $\lambda_1(G)>0\geq\lambda_2(G)$, and
 $\lambda_{n-\theta+1}(G)=\lambda_{n-\theta+2}(G)=\cdots=\lambda_n(G)$, that is, $G$ has two or three distinct eigenvalues with the positive inertia 1.
 By Lemmas \ref{le2} and \ref{le4} with the fact that only complete
 graphs have two distinct eigenvalues, the required result follows
 immediately.
  \end{proof}

 By setting $t=2$ in Theorem \ref{th3}, we get a stronger result for $k\geq 2$.
 \begin{theorem} \label{coro2} Let $G$ be a graph of order $n\geq 2$ with negative inertia $\theta$. If $G\ncong K_n$, then for every integer $k\geq 2$,
 \begin{equation*}
 \sum_{i=1}^{k}\lambda_i(G)\leq \sqrt{2\,m\left(k-\frac{4}{\theta+2}\right)}
 \end{equation*}
 with equality holding if and only if $G\cong \,K_{p,\,q}\cup K_{s,t}$ with $p+q+s+t=n, pq=st$ and $k=2$ or $G\cong 2\,K_{\underbrace{t,\,t,\ldots,\,t}_p}$ $(n=2\,p\,t,\,p\geq 3)$ with $k=2$.
 \end{theorem}
 
 \begin{proof}
 If $\theta=n-1$, it is clear that $G\cong K_n$, which is not in our consideration. Otherwise, $\theta\leq n-2$.
 Then we set $t=2$ and $M=A(G)$ in Theorem \ref{th3}, we get
    $$\sum_{i=1}^{k}\lambda_i(G)\leq\sqrt{2m\,\Big(k-\frac{4}{\theta+2}\Big)}~~\mbox{for }k\geq 2.$$
 The equality holds if and only if $k=2$ with $\lambda_1(G)=\lambda_2(G)>0\geq\lambda_3(G)$ and $\lambda_{n-\theta+1}(G)=\lambda_{n-\theta+2}(G)=\cdots=\lambda_n(G)$.
 Now it remains to characterize the graphs with $\lambda_1(G)=\lambda_2(G)>0\geq\lambda_3(G)$ and $\lambda_{n-\theta+1}(G)=\lambda_{n-\theta+2}(G)=\cdots=\lambda_n(G)$.
 By Perron-Frobenius theory, then $\lambda_1(G)$ is an eigenvalue of $G$ with multiplicity $1$ if $G$ is connected.
 As $\lambda_1(G)=\lambda_2(G)>0\geq\lambda_3(G)$, we conclude that $G$ has two connected
 components, say $G_1$ and $G_2$. Then both $G_1$ and $G_2$ have two
 or three distinct eigenvalues with positive inertia $1$. Then by Lemmas \ref{le2} and
 \ref{le4},  $G_1$ and $G_2$ are complete bipartite graphs or complete multipartite graphs with all equal partitions.
  As $S_A(K_{p,\,q})=\Big\{\sqrt{pq},\,0,\,\ldots,\,0,\,-\sqrt{pq}\Big\}$
  with the result given in (\ref{sunm2}), we get the required result after
  simple calculation. This completes the proof.
  \end{proof}
  As the maximum size of any planar graph (triangle-free graph) of order
  $n$ is $3\,n-6$ ($\frac{n^2}{4}$), by setting $k=2$ in Theorem \ref{coro2} with the fact that
  a non-complete graph of order $n$ has the negative inertia at most $n-2$,
  we obtain the following result immediately.

  \begin{corollary} \label{coro3}
  $(i)$ Let $G$ be a planar graph of order $n\geq 3$. Then  
   $$\lambda_1(G)+\lambda_2(G)<\frac{2\,\sqrt{3}}{\sqrt{n}}(n-2)<n.$$
  $(ii)$ Let $G$ be a triangle-free graph of order $n\geq 3$. Then  
    $$\lambda_1(G)+\lambda_2(G)< \sqrt{n\,(n-2)}<n.$$
  \end{corollary}

  \begin{theorem} \label{nth1} {\rm \cite{NIK}} Let $G$ be a graph of order
   $n$. Then $\lambda_1(G)+\lambda_2(G)\leq \frac{2\,n}{\sqrt{3}}$.
 \end{theorem}

   \begin{remark} {\rm The results in Theorem \ref{coro2} with $k=2$ and Theorem \ref{nth1} are
   incomparable. Corollary \ref{coro3} gives the two classes of
   graphs in which the result in Theorem \ref{coro2} is better.}
 \end{remark}
 Next, we present an application of Theorem \ref{th4}. Although this theorem demands strong preconditions, the adjacency matrix of a graph inherently satisfies its spectral conditions. Substituting the facts that $tr(A)=0$ and $tr(A^2)=2m$ into the theorem with $s=1$, and invoking Lemmas \ref{le1} and \ref{le2}, we directly derive the following corollary.
 \begin{corollary} {\rm \cite{DMS}}\label{nco1} Let $G$ be a graph of order $n\geq 2$ with negative inertia $\theta$. Then for every integer $k\leq n$,
 \begin{equation*}
 \sum_{i=1}^{k}\lambda_i\leq \frac{n}{2\left(\theta+1\right)}\left(\theta+\sqrt{\theta\left(k\,\theta+k-1\right)\,\,}~\right)
 \end{equation*}
 with equality holding if and only if $G\cong K_{\underbrace{t,\,t,\ldots,\,t}_p}$ $(n=p\,t,\,p\geq 2)$ with $k=1$.
 \end{corollary}

\begin{remark}{\rm 
According to \cite{DMS}, the upper bound for $\sum_{i=1}^{k} \lambda_i$ of a graph established in Corollary \ref{nco1} is sharper than the bound $\frac{n}{2}(\sqrt{k} + 1)$ given by Mohar \cite{MOHAR}.}
\end{remark}

  \subsection{Laplacian matrix}
 This subsection is dedicated to applying Theorem \ref{th3} to the Laplacian matrix, yielding several new bounds on the sum of the $k$ largest eigenvalues. These results advance the resolution of Brouwer's conjecture by offering partial solutions through different theoretical avenues.
 Now we give the first result on Laplacian spectrum by applying Theorem \ref{th3}.
  \begin{theorem}\label{thm31}
   Let $G$ be a graph of order $n$. For every integer $k$, we have
 \begin{equation}
 \sum_{i=1}^{k}\mu_i(G)\leq\frac{2mk}{n}+\sqrt{\left(\sum_{i=1}^{n}d_i^2+2m-
 \frac{4m^2}{n}\right)\,\left(k-\frac{1}{\theta+1}\right)} \label{sunm3}
 \end{equation}
 with equality holding if and only if $G\cong K_{n/2,\,n/2}$ $(n$ is even$)$ and
 $k=1$.
 \end{theorem}
 
 \begin{pf}
 By setting $t=1$ and $M=L(G)$ in Theorem \ref{th3} with the facts that
 $tr(L(G))=2m$ and $tr((L(G)-\frac{2m}{n}I)^2)=\sum_{i=1}^{n}d_i^2+2m-
 \frac{4m^2}{n}$, we get the result in (\ref{sunm3}) with equality holding if and only if $k=1$ with $\mu_1(G)>\frac{2m}{n}=\mu_2(G)=\cdots=\mu_{n-\theta}$ and
 $\mu_{n-\theta+1}(G)=\mu_{n-\theta+2}(G)=\cdots=\mu_n(G)$, that is, $G$ has three distinct
 Laplacian eigenvalues with $\mu_2(G)=\frac{2m}{n}$.
 By Lemma \ref{newle1}, we conclude that $G\cong K_{n/2,\,n/2}$ ($n$ is even), which completes the proof.
 \end{pf}

 \begin{remark} {\rm One can easily confirm that  
 $$tr\left(\Big(L(G)-\frac{tr(L(G))}{n}I\Big)^2\right)=tr\left(\Big(Q(G)-\frac{tr(Q(G))}{n}I\Big)^2\right),$$ 
 where $Q(G)$ is the signless Laplacian matrix of $G$. Thus, the result in (\ref{sunm3}) also holds for $Q(G)$. Moreover, the result in (\ref{sunm3}) can be improved by setting $t=2$ (Lemma \ref{newle1} ensures that $t$ can be $2$) in Theorem \ref{th3} except star graphs.}
 \end{remark}

Next we present the results on the characterization of graphs with three or four distinct Laplacian eigenvalues, subject to specified conditions.
\begin{lemma} \label{nlesundas1}
    Let $G$ be a connected graph of order $n\geq 3$ and size $m$. Then $G$ has three distinct Laplacian eigenvalues with $\mu_1=n$ and $\mu_2=\mu_{n-1}<\frac{2m+n}{n}$
    if and only if
    $G\cong K_{1,\,n-1}$ or $G\cong K_{\frac{n}{2},\,\frac{n}{2}}$ with even $n$.
    \end{lemma}
    
\begin{pf}
 We first prove the sufficiency. From the given condition with Lemma \ref{newle2}, we have
 \begin{align}
    S_L(\overline{G})=\Big\{\underbrace{n-\mu_{n-1}(G),\,\ldots,\,n-\mu_{n-1}(G)}_{n-2},\,0,\,0\Big\}\label{kin12}
 \end{align}
 with $\mu_2(G)=\mu_{n-1}(G)<\frac{2m+n}{n}$. If $G\cong K_n$, then $S_L(G)=\Big\{\underbrace{n,\ldots,\,n}_{n-1},\,0\Big\}$, a contradiction as $G$ has three distinct Laplacian eigenvalues. Otherwise, $G\ncong K_n$. Thus
 $\mu_{n-1}(G)<n$ and hence $n-\mu_{n-1}(G)>0$. From (\ref{kin12}), we conclude that $\overline{G}$ is disconnected and has exactly two connected
 complements, say $H_1$ of order $n_1$  and $H_2$ of order $n_2$ such
 that $n_1+n_2=n$, where $n_1\geq n_2$. Then
 $S_L(\overline{G})=S_L(H_1)\cup S_L(H_2)$. Then $H_1\cong K_{n_1}$
 and $H_2\cong K_{n_2}$ due to the result in (\ref{kin12}) and the well-known fact that a connected graph with exactly two distinct
 Laplacian eigenvalues is a complete graph. For $n_2=1$, we obtain $H_2\cong K_1$ and $H_1\cong K_{n-1}$, that is, $G\cong K_{1,n-1}$.
 We now assume that $n_2\geq 2$. Therefore, $H_1$ and $H_2$ both have two distinct Laplacian eigenvalues whose nonzero eigenvalues are the same.
 Due to the above results, we have $H_1\cong K_{n/2}$ and $H_2\cong K_{n/2}$ $(n$ is even$)$, which follows that $G\cong K_{n/2,\,n/2}$ $(n$ is even$)$.

 Next we prove the necessity. If $G\cong K_{1,\,n-1}$, then $\frac{2m+n}{n}=\frac{3n-2}{n}>1$
 and $S_L(G)=\Big\{n,\,\underbrace{1,\ldots,\,1}_{n-2},\,0\Big\}$, which
 satisfies the required conditions. Otherwise, $G\cong
 K_{\frac{n}{2},\,\frac{n}{2}}$.  Then $\frac{2m+n}{n}=\frac{n+2}{2}>\frac{n}{2}$
 and $S_L(G)=\Big\{n,\,\underbrace{\frac{n}{2},\ldots,\,\frac{n}{2}}_{n-2},\,0\Big\}$,
 again satisfying the required conditions.
\end{pf}

Next we  define a set of graphs as follows:
\begin{itemize}
 \item For positive integers $n$ and $t$, $\mathcal{Q}(n,t)$ denotes the class of connected graphs of order $n$ and size $\frac{(n+1)t}{2}$ with exactly three distinct Laplacian eigenvalues and algebraic connectivity equal to $t$. 
\end{itemize}
One can easily confirm that  $\mathcal{Q}(n,\,n-2)=\{K_{n-2}\vee 2K_1\}$.
Moreover, let $G\cong \overline{aK_b\cup cK_1}$, where $a,\,b,\,c$ are positive integers, $a\geq 2$ and $b\geq 3$. By considering the Laplacian spectrum of $\overline{G}$, we obtain that $G$ has two distinct nonzero Laplacian eigenvalues: $ab+c$ and $(a-1)b+c$. In order to make the graph $G$ belongs to 
$\mathcal{Q}(ab+c,\,(a-1)b+c)$, then $a,\,b,\,c$ must satisfy the following equation: 
$$(ab+c)(a+c-1)+(ab+c-b)(ab-a)=(ab+c+1)(ab+c-b),$$
that is, $c=\frac{b(a-1)}{b-2}$. Hence $\overline{aK_b\cup cK_1}\in \mathcal{Q}(ab+c,\,(a-1)b+c)$ with $a,\,b,\,c$ are positive integers, $a\geq 2$, $b\geq 3$ and $c=\frac{b(a-1)}{b-2}$. For example, by taking $b=3$, we have 
$\overline{aK_3\cup (3a-3)K_1}\in \mathcal{Q}(6a-3,\,6a-6)$.
 For convenience, let
 \begin{align*}
 \mathcal{G}_1(n)&=\Big\{K_1\vee\overline{H}|~H\in \mathcal{Q}(n-1,\,t),\,2\leq t\leq n-3\Big\},\\[2mm]
 \mathcal{G}_2(n)&=\Bigg\{\overline{H}\vee tK_1|~H\in \mathcal{Q}(n-t,\,t)~~\mbox{and integer } t\in \Big[2,\,\frac{n}{2}\Big]\Bigg\},\mbox{ and }\\[3mm]
 \mathcal{G}_3(n)&=\Big\{yK_x\vee yK_x|~\mbox{$x$ and $y$ are positive integers such that $2xy=n$, $n$ is even and $y>1$} \Big\}.
 \end{align*}

\begin{lemma} \label{nlesundas2}
    Let $G$ be a connected graph of order $n\geq 4$ and size $m$. Then $G$ has four distinct Laplacian eigenvalues with $\mu_1=n$ and $\mu_2=\frac{2m+n}{n}$ if and only if
    $G\in \mathcal{G}_1(n)\cup\mathcal{G}_2(n)\cup\mathcal{G}_3(n)$ except $K_{\frac{n}{2},\,\frac{n}{2}}$.
    \end{lemma}
    
\begin{pf}
 We first prove the sufficiency. From the given condition, we have
 $$S_L(G)=\left\{n,\,\underbrace{\frac{2m+n}{n},\,\ldots,\,\frac{2m+n}{n}}_{n-\ell-2},\,\underbrace{\mu_{n-1}(G),\ldots,\,\mu_{n-1}(G)}_{\ell},\,0\right\}, ~~1\leq \ell\leq n-3$$
 with $\mu_{n-1}(G)<\frac{2m+n}{n}$. By Lemma \ref{newle2} with the above result, we have
\begin{equation}
    S_L(\overline{G})=\left\{\underbrace{n-\mu_{n-1}(G),\,\ldots,\,n-\mu_{n-1}(G)}_{\ell},\,\underbrace{\frac{n^2-2m-n}{n},\ldots,\,\frac{n^2-2m-n}{n}}_{n-\ell-2},\,0,\,0\right\}.\label{sunm5}
\end{equation}
Since $\mu_2(G)<n$, from (\ref{sunm5}), we conclude that
$\overline{G}$ is disconnected and has exactly two connected
complements, say $H_1$ of order $n_1$ and size $m_1$, and $H_2$ of
order $n_2$ such that $n_1+n_2=n$, where $n_1\geq n_2$. Then
$S_L(\overline{G})=S_L(H_1)\cup S_L(H_2)$. Depending on the value of
$n_2$, we now consider the following two cases:

\vspace*{3mm}

\noindent ${\bf Case\,1.}$ $n_2=1$. Then $G\cong K_1\vee
\overline{H}_1$. In this case, $\mu_{n-1}(G)\geq 1$. From
(\ref{sunm5}), we obtain
\begin{align}
    S_L(\overline{H_1})=\left\{\underbrace{\frac{2m}{n},\,\ldots,\,\frac{2m}{n}}_{n-\ell-2},\,\underbrace{\mu_{n-1}(G)-1,\ldots,\,\mu_{n-1}(G)-1}_{\ell},\,0\right\},\nonumber
\end{align}
$$~~\mbox{that is, }
    S_L(H_1)=\left\{\underbrace{n-\mu_{n-1}(G),\ldots,\,n-\mu_{n-1}(G)}_{\ell},\,\underbrace{n-1-\frac{2m}{n},\,\ldots,\,n-1-\frac{2m}{n}}_{n-\ell-2},\,0\right\},$$
where $\mu_1(H_1)=n-\mu_{n-1}(G)>n-1-\frac{2m}{n}=\mu_{n-2}(H_1)$.
 As $2m=n(n-1)-2m_1$, we get $2m_1=n\,\mu_{n-2}(H_1)$. Therefore, $H_1\in \mathcal{Q}(n-1,\,t)$ with $2\leq t\leq n-2$, which implies that $G\in \mathcal{G}_1(n)$.

\vspace*{3mm}

\noindent ${\bf Case\,2.}$ $n_2\geq 2$. If $H_1$ and $H_2$  are both
complete, then $G\cong K_{n_1,\,n_2}$. We have
 $$\mu_2(G)=n_1=\frac{2n_1n_2}{n_1+n_2}+1,$$
 which implies that $n_2=n_1-2+\frac{2}{n_1+1}$. As $n_1$ and $n_2$
 are both positive integers, then $n_1=1$ and $n_2=0$, a contradiction.
 Otherwise, at least one of $H_1$ and $H_2$ is not complete.

\vspace*{3mm}

\begin{claim}\label{c2} $G$ is regular or  $G\cong \overline{H}_1\vee n_2K_1$, where $H_1\in \mathcal{Q}(n_1,n_2)$.
\end{claim}

\vspace*{3mm}

\noindent {\bf Proof of Claim \ref{c2}.} Since at least one of $H_1$ and $H_2$ is not complete, we consider the following three cases:

\vspace*{2mm}

\noindent
${\bf Case\,2.1.}$ Both $H_1$ and $H_2$ are not complete. Thus we have $\mu_{n_i-1}(H_i)\leq\delta(H_i)$
 for $i=1,\,2$ by Lemma \ref{newsundas}. Then
    $$\frac{2m+n}{n}=\mu_2(G)=n-\min\Big\{\mu_{n_1-1}(H_1),\,\mu_{n_2-1}(H_2)\Big\}\geq n-\min\Big\{\delta(H_1),\,\delta(H_2)\Big\}=\Delta(G)+1,$$
that is, $2m\geq n\,\Delta$ and hence $2m=n\,\Delta$. Thus $G$ is regular.

 \vspace*{3mm}

\vspace*{2mm}

\noindent
${\bf Case\,2.2.}$ $H_1\cong K_{n_1}$. Then $H_2$ is not complete. This with Lemma \ref{newsundas}, we obtain
 \begin{eqnarray}
 \mu_2(G)&=&\frac{2m}{n}+1=n-\min\Big\{\mu_{n_1-1}(H_1),\,\mu_{n_2-1}(H_2)\Big\}=n-\min\Big\{n_1,\,\mu_{n_2-1}(H_2)\Big\}\nonumber\\[2mm]
 &\geq & n-\min\{n_1,\,\delta(H_2)\}\nonumber\\[2mm]
 &=&n-\delta(H_2) =n-\delta(\overline{G})~~~\mbox{as $n_1\geq n_2$}\nonumber\\[2mm]
 &=&\Delta(G)+1.\nonumber
 \end{eqnarray}
 Hence $G$ is regular.

 \vspace*{3mm}

\vspace*{2mm}

\noindent
${\bf Case\,2.3.}$ $H_2\cong K_{n_2}$. Thus we have $H_1\ncong K_{n_1}$. If $\mu_{n_1-1}(H_1)>n_2$, then we have
 $\mu_{1}(H_1)=\mu_{n_1-1}(H_1)$ as $S_L(\overline{G})=S_L(H_1)\bigcup S_L(H_2)$ and the result in (\ref{sunm5}).
 This follows that $H_1$ is complete, a contradiction. Otherwise, $\mu_{n_1-1}(H_1)\leq n_2$. In this case, we have
 \begin{equation}
 \frac{2m}{n}+1=\mu_2(G)=n-\min\{\mu_{n_1-1}(H_1),\,n_2\}=n-\mu_{n_1-1}(H_1).\label{smdas}
 \end{equation}
  Now we first assume that $\mu_{n_1-1}(H_1)\leq \delta(\overline{G})$. By (\ref{smdas}), we have
  $\frac{2m}{n}+1\geq n-\delta(\overline{G})=\Delta(G)+1$, which implies that $G$ is regular.
  Next we assume that $\mu_{n_1-1}(H_1)> \delta(\overline{G})$. By Lemma
  \ref{newsundas} , we have $\delta(H_1)\geq
  \mu_{n_1-1}(H_1)>\delta(\overline{G})=\min\{\delta(H_1),\,n_2-1\}$.
  It follows that $\delta(H_1)>\delta(\overline{G})=n_2-1$ and hence $\delta(H_1)\geq
  n_2$. Then we also have $\mu_1(H_1)\geq \Delta(H_1)+1\geq n_2+1$
  by Lemma \ref{newle3}. Since $S_L(\overline{G})=S_L(H_1)\bigcup
  S_L(H_2)$ contains two nonzero distinct entries and $n_2\in
  S_L(H_2)$ with the above results, then $\mu_{n_1-1}(H_1)=n_2$. By (\ref{smdas}), we have
  $$n_1=n- \mu_{n_1-1}(H_1)=\frac{2m}{n}+1=\frac{n_1(n_1-1)-2m_1+2n_1n_2}{n_1+n_2}+1,$$
  that is, $2m_1=(n_1+1)n_2$. Hence $H_1\in \mathcal{Q}(n_1,\,n_2)$.
  This completes the proof of  ${\bf Claim\,\ref{c2}.}$

\vspace*{3mm}

 \noindent
 By ${\bf Claim\,\ref{c2}}$ with facts that $n_1\geq n_2\geq 2$, $G\in \mathcal{G}_2(n)$ or $G$ is regular. Thus it remains to prove for regular graph $G$. Let $G$ be $r$-regular.
 From this, we conclude that $H_1$ and $H_2$ both are $(n-r-1)$-regular graphs. For convenience, let $q=\mu_{n-1}(G)$. By (\ref{sunm5}), we have
  $$ S_L(\overline{G})=\Bigg\{\underbrace{n-q,\,\ldots,\,n-q}_{\ell},\,\underbrace{n-r-1,\,\ldots,\,n-r-1}_{n-\ell-2},\,0,\,0\Bigg\},$$
  where $1\leq \ell\leq n-3$. As $\overline{G}$ is regular, then
  $$ S_A(\overline{G})=\Bigg\{n-r-1,\,n-r-1,\,\underbrace{0,\,\ldots,\,0}_{n-\ell-2},\,\underbrace{q-r-1,\,\ldots,\,q-r-1}_{\ell}\Bigg\}.$$
 Since $\overline{G}\cong H_1\bigcup H_2$, then $\lambda_2(H_i)\leq
 0$ for $i=1,\,2$. Note that $\lambda_2(H_1)\lambda_2(H_2)=0$ as $\ell\leq
 n-3$. Without loss of generality, we assume that $\lambda_2(H_1)\geq
 \lambda_2(H_2)$. If $\lambda_2(H_1)>\lambda_2(H_2)$, then
 $\lambda_2(H_1)=0$ and $\lambda_2(H_2)=q-r-1<0$. In this case, $H_1$ has three distinct eigenvalues
 and $H_2$  has two distinct eigenvalues, which follows $H_2\cong K_{n_2}$. This implies
 $\lambda_{n_2}(H_2)=-1=q-r-1$, that is, $q=r$.
 By Lemmas \ref{le2} and \ref{le4} with the fact that $H_1$ is regular, we obtain $H_1\cong
 K_{\underbrace{t,\,\ldots,\,t}_s}$ with $s,\,t\geq 2$. By
 (\ref{sunm2}), we have $\lambda_{n_1}(H_1)=-t=q-r-1=-1$ as $q=r$,
 a contradiction. Otherwise, $\lambda_2(H_1)=\lambda_2(H_2)=0$.
 Similarly, we obtain that $H_1$ and $H_2$ both are regular complete
 multipartite graphs by Lemmas \ref{le2} and \ref{le4}. As they have
 same nonzero eigenvalues with the result in (\ref{sunm2}), we
 conclude that $H_1\cong H_2\cong K_{\underbrace{x,\,\ldots,\,x}_y}$ with $x,\,y\geq
 2$ and $2xy=n$, that is, $G\in  \mathcal{G}_3(n)\setminus \{\frac{n}{2}K_1\vee \frac{n}{2}K_1\}$. This completes the proof of the sufficiency.

\vspace*{3mm}

\noindent
 Next we prove the necessity by showing the graphs from $\mathcal{G}_i(n)$ satisfy (\ref{sunm5}) for $i=1,\,2,\,3$ except $\frac{n}{2}K_1\vee \frac{n}{2}K_1$.
 For $G\in \mathcal{G}_1(n)$, then $\overline{G}=K_1\cup H$, where
 $H\in \mathcal{Q}(n-1,\,t)$ with $2\leq t\leq n-3$. By the
 definition of $\mathcal{Q}(n-1,\,t)$, $H$ is connected, and has order $n-1$ and size
 $\frac{nt}{2}$ with three distinct Laplacian eigenvalues and
 algebraic connectivity $t$. Thus $\overline{G}$ has three distinct
 Laplacian eigenvalues with two zero Laplacian eigenvalues.
 As $\overline{G}=K_1\cup H$, we have
 $n(n-1)-2m=nt$, that is, $t=\frac{n^2-n-2m}{n}$. Therefore, the Laplacian spectrum of
 $\overline{G}$ satisfies the result in (\ref{sunm5}).

 For $G\in \mathcal{G}_2(n)$, then $\overline{G}=K_t\cup H$, where
 $H\in \mathcal{Q}(n-t,\,t)$ with $2\leq t\leq \frac{n}{2}$. Again
 by the definition of $\mathcal{Q}(n-t,\,t)$, we conclude that $\overline{G}$ has three distinct
 Laplacian eigenvalues with two zero Laplacian eigenvalues and the
 smallest nonzero Laplacian eigenvalue equal to
 $t=\frac{n^2-n-2m}{n}$, which satisfies the result in (\ref{sunm5}).

  For $G\in \mathcal{G}_3(n)\setminus \{\frac{n}{2}K_1\vee \frac{n}{2}K_1\}$, then
  $\overline{G}=2K_{\underbrace{x,\,\ldots,\,x}_y}$ with $x,\,y\geq 2$ and
  $2xy=n$. Note that $m=x^2y^2+x(x-1)y$ and
  $S_L(K_{\underbrace{x,\,\ldots,\,x}_y})=\{\underbrace{xy,\ldots,\,xy}_{y-1},\,\underbrace{xy-x,\ldots,\,xy-x}_{xy-y},\,0\}$.
 It implies that $\overline{G}$ has three distinct
 Laplacian eigenvalues with two zero Laplacian eigenvalues and the
 smallest nonzero Laplacian eigenvalue equal to
 $\frac{n^2-n-2m}{n}$ as
 $$\frac{n^2-n-2m}{n}=\frac{4x^2y^2-2xy-2xy(xy+x-1)}{2xy}=xy-y$$ 
 and $x,\,y\geq 2$, again satisfying the result in (\ref{sunm5}). This completes
 the proof.
\end{pf}

Next, we establish another upper bound on the sum of the $k$ largest Laplacian eigenvalues by applying Theorem \ref{th3}. Notably, this bound is attained by  a wide class of graphs.
 \begin{theorem} \label{thm32}
    Let $G$ be a graph of order $n\geq 2$ and let $\sigma$ be the number of Laplacian eigenvalues of $G$ less than
    $\frac{2m}{n}+1$. Then for any integer $k$, we have
 \begin{equation}
    \sum_{i=1}^{k}\mu_i(G)\leq\frac{(2m+n)(k+1)}{n}-n+\sqrt
    {\left(\sum_{i=1}^{n}d^2_i+n^2-n-2m-\frac{4m^2}{n}\right)\,\left(k+1-\frac{4}{\sigma+1}\right)} \label{sunm4}
 \end{equation}
 with equality holding if and only if $G\cong K_n$ or $k=1$ with $G\cong K_{1,\,{n-1}}$ or $k=1$ with $G\in \bigcup_{i=1}^3\mathcal{G}_i(n)$.
      \end{theorem}
  \begin{proof}
 Let $M=L+J$. It is clear that $\lambda_1(M)=n$ and
 $\lambda_i(M)=\mu_{i-1}(G)$ for $2\leq i\leq n$, which follows that
 $\theta(M)=\sigma-1$ and $\lambda_2(M)=\mu_{1}(G)\geq \frac{2m}{n}+1$ (as $\mu_{1}(G)\geq \Delta+1$). Moreover, after simple calculation, we have
$$tr\left((M-\frac{tr(M)}{n}I)^2\right)=\sum_{i=1}^{n}d^2_i+n^2-n-2m-\frac{4m^2}{n}.$$
  By setting $t=2$ and $M=L+J$ in Theorem \ref{th3} and combining the above facts, we get
  \begin{equation*}
 \sum_{i=1}^{k+1}\lambda_i(M)\leq \frac{(2m+n)(k+1)}{n}+\sqrt
    {\left(\sum_{i=1}^{n}d^2_i+n^2-n-2m-\frac{4m^2}{n}\right)\left(k+1-\frac
    {4}{\theta(M)+2}\right)}
 \end{equation*}
 with equality holding if and only if $M=\frac{2m+n}{n}I$ or $k=1$
 with $\lambda_1(M)=\lambda_2(M)>\frac{2m}{n}+1\geq\lambda_3(M)$
    and $\lambda_{n-\theta(M)+1}(M)=\cdots=\lambda_{n}(M)$, that is,
 the result in (\ref{sunm4}) holds with
 the equality holding if and only if $G\cong K_n$, or $k=1$ with
$$\mbox{$\mu_1(G)=n>\frac{2m}{n}+1=\mu_2(G)=\cdots=\mu_{n-\sigma}(G)$ and $\frac{2m}{n}+1>\mu_{n-\sigma+1}(G)=\cdots=\mu_{n-1}(G)$}$$
or
$$\mbox{$\mu_1(G)=n>\frac{2m}{n}+1>\mu_2(G)=\cdots=\cdots=\mu_{n-1}(G)$}.$$
 By Lemmas \ref{nlesundas1} and \ref{nlesundas2}, we get required extremal graphs.
 This completes the proof of this theorem.
  \end{proof}

Brouwer's conjecture has been confirmed for all graphs of order up to $10$ and is true for $k\in\{1, 2, n-3, n-2, n-1, n\}$ {\rm \cite{D,F}}.
   In addition, Conjecture \ref{conj1} has also been proven to hold for certain classes of graphs,
   such as trees{\rm \cite{F}}, unicyclic graphs{\rm \cite{G,H}}, bicyclic graphs{\rm \cite{G}},
   threshhold graphs{\rm \cite{F}}, regular graphs{\rm \cite{I}}, cographs{\rm \cite{I}} and split graphs{\rm \cite{I}}.
   However, it is still open for general graphs. Next we will apply the above results to give a partial answer on this conjecture. 


\begin{theorem}\label{th6}
  Let $G$ be a simple graph on $n$ vertices and $m$ edges. Then the Brouwer's conjecture holds
  for $m\geq (k+1)n$ and $k=cn^{\alpha}+O(1)$ with sufficiently large $n$, where $c$ and $\alpha$ satisfy one of the following
  conditions:\\
 $(i)$ $0\leq\alpha<\frac{1}{2}$ and $c>0$;\\
 $(ii)$ $\alpha=\frac{1}{2}$ and $0<c<1$;\\
 $(iii)$ $\alpha=1$ and $c>\frac{4}{9}$.
\end{theorem}

\begin{proof}
 By Theorem \ref{thm31} with Lemma \ref{2le1} and fact that $\sigma\leq
 n-1$, we have
 $$ \sum_{i=1}^{k}\mu_i(G)\leq\frac{(2m+n)(k+1)}{n}-n+\sqrt
    {\left(n^2+(n-4)m+\frac{2m^2}{n-1}-n-\frac{4m^2}{n}\right)\left(k+1-\frac
    {4}{n}\right)}.$$
 Now we will prove Brouwer's conjecture under the conditions in
 this theorem by showing that the above bound  is less than
 $m+\frac{(k+1)k}{2}$, that is, $f(m)>0$, where

\begin{multline*}
  f(m)=\left(m+n-1+\frac{k(k-1)}{2}-\frac{2m(k+1)}{n}\right)^2\\
  -\left(\frac{2m^2}{n-1}+(n-4)m+n^2-n-\frac{4m^2}{n}\right)\left(k+1-\frac{4}{n}\right).
\end{multline*}
 After simple calculation, we rewrite $f(m)=\frac{a_1m^2+a_2m+a_3}{4(n-1)n^2}$, where
\begin{align*}
a_1&=4n^3-(8k+12)n^2+(16k^2+32k-16)n-16k^2-32k+48,\\
a_2&=(-4k+4)n^4+(4k^2+4)n^3-(8k^3+4k^2-28k+56)n^2+(8k^3-24k+48)n,\\
a_3&=-4kn^5+(4k^2+4k+12)n^4+(k^4-2k^3-7k^2+4k-24)n^3\\
&~~~~~~~~~~~~~~~~~~~~~-(k^4-2k^3-3k^2+4k-12)n^2.
\end{align*}
 From this, we have
\begin{multline}
 f(kn+n)(4n-4)=(4k+8)n^3-(4k^3+20k^2+24k-4)n^2+(9k^4+ 50k^3+ \\81k^2- 24k -96)n- 9k^4 - 54k^3- 53k^2 + 84k +
 108. \label{sunm6}
\end{multline}

\noindent
${\bf Condition\,(i).}$  $0\leq\alpha<\frac{1}{2}$ and $c>0$. In this case,  $a_1=4n^3+O(n^{2+\alpha})>0$, $a_2=(4-4k)n^4+O(n^{3+2\alpha})<0$ and
$a_3=-4kn^5+O(n^{4+2\alpha})<0$. It follows that $f(m)$ has one positive root and one negative root. By (\ref{sunm6}), we obtain 
     $$f(kn+n)=\frac{4cn^{\alpha+3}+O(n^{3\alpha+2})}{4(n-1)}>0$$
as $c>0$ and $n$ is sufficiently large. This implies that the positive root of $f(m)$ is smaller than $(k+1)n$. Then $f(m)$ is
increasing on $m\geq (k+1)n$. Hence $f(m)>0$ for $m\geq (k+1)n$.

\vspace*{3mm}

\noindent
${\bf Condition\,(ii).}$  $\alpha=\frac{1}{2}$ and $0<c<1$. Then $a_1=4n^3+O(n^{2.5})>0$, $a_2=-4cn^{4.5}+O(n^{4})<0$ and
$a_3=-4cn^{5.5}+O(n^{5})<0$. By (\ref{sunm6}), we obtain
   $$f(kn+n)=\frac{4c(1-c^2)n^{3.5}+O(n^3)}{4(n-1)}>0$$ 
as $0<c<1$. Similarly, we confirm that $f(m)>0$ for $m\geq (k+1)n$.

\vspace*{3mm}

\noindent
${\bf Condition\,(iii).}$ $\alpha=1$ and $\frac{4}{9}<c$. Then $a_1=(16c^2-8c+4)n^3+O(n^2)>0$ and $a_2=-4c(2c^2-c+1)n^5+O(n^4)<0$. Then we have
\begin{align*}
    (k+1)n+\frac{a_2}{2a_1}=&cn^2+O(n)+\frac{-4c(2c^2-c+1)n^5+O(n^4)} {8(4c^2-2c+1)n^3+O(n^2)}\\[3mm]
    =&cn^2+\frac{-c(2c^2-c+1)} {2(4c^2-2c+1)}n+O(n)\\[3mm]
    =&\frac{(6c^3-3c^2+c)}{(8c^2-4c+2)}n^2+O(n)>0
\end{align*}
as $c>0$. Moreover, from (\ref{sunm6}), we get
   $$f(kn+n)=\frac{c^3(9c-4)n^5+O(n^4)}{4(n-1)}>0$$ 
as $c>\frac{4}{9}$. Hence we conclude that $f(m)>0$ for $m\geq (k+1)n$. This completes the proof.
\end{proof}

 The above theorem implies that Brouwer's conjecture is true for
 almost all graphs when $k$ is small. We now use same idea in the
 proof of the above theorem to give the results on $k=3,\,4$.
\begin{corollary}\label{coro5}
  Let $G$ be a simple graph with size $m$ and order $n$.\\
  $(i)$ If $m\geq 3n+9$ and $n\geq 4$, then Brouwer's conjecture holds for
  $k=3$;\\
  $(ii)$ if $m\geq 4n+17$ and $n\geq 6$, then Brouwer's conjecture holds for
  $k=4$.
\end{corollary}

\begin{proof}
  The proof is similar to the proof of Theorem \ref{th6}. We consider the function $f(m)$ defined in the proof of Theorem \ref{th6}.
   For $k=3$, we have
  \begin{equation*}
   f(m)=\frac{(n^2-8n+48)m^2-(2n^3-8n^2+48n)m-3n^4+12n^3}{n^2}.
   \end{equation*}
 As $n\geq 4$, we have $n^2-8n+48>0$, $-3n^4+12n^3\leq 0$, $f(3n+9)=\frac{9(n^2+168n+432)}{n^2}>0$.
 Thus $f(m)>0$ for $m\geq 3n+9$ with $n\geq 4$, that is, Brouwer's conjecture holds in this
 case. This completes the proof of (i). For $k=4$, similarly we get
 $f(m)$ has one positive root and one negative root, and
 $f(4n+17)>0$ for $n\geq 6$. Hence the result of (ii) follows.
\end{proof}

\begin{remark}{\rm
  Theorem 3.1 in \cite{D} has proved that if  Brouwer's conjecture holds for all graphs with $k=p$,
  then it also holds for all graphs with $k=n-p-1$. By a similar proof of this theorem, we easily obtain that if
  Brouwer's conjecture holds for all graphs with $m\geq s$ and $k=p$, then it also holds for all graphs with $m\leq\frac{n(n-1)}{2}-s$ and $k=n-p-1$.
  Then the results in Theorem \ref{th6} and Corollary \ref{coro5}
  can be rewritten in another way. For example, by the (i) of Corollary
  \ref{coro5}, we get that Brouwer's conjecture holds for $k=n-4$ if
  $m\leq \frac{n^2-7n-18}{2}$ and $n\geq 4$.}
\end{remark}

\section{Concluding remark}
In this paper, we have investigated the upper bound for the sum of
the $k$ largest eigenvalues of symmetric matrices, as presented in
Theorems \ref{th3} and \ref{th4}. Applying these results to the
adjacency and Laplacian matrices of a graph, we obtained several
interesting estimates. In particular, we established new upper
bounds on the sum of the $k$ largest Laplacian eigenvalues. These
bounds imply that Brouwer's conjecture holds for small $k$ in the
case of almost all graphs, thereby contributing a meaningful step
toward its complete resolution. It is important to emphasize that the results derived in Section 2 possess broad applicability beyond the matrices discussed here; they can be directly applied to other graph matrices to yield novel spectral bounds by combining their properties.

 \end{document}